\documentclass[12pt]{amsart}

\usepackage[utf8]{inputenc}

\usepackage{amsmath}
\usepackage{amsfonts}
\usepackage{amssymb,enumerate}
\usepackage{amsthm}
\usepackage{fancyhdr}
\usepackage{hyperref}
\usepackage{cleveref}
\usepackage{xcolor}
\usepackage{enumitem}
\usepackage{float}
\usepackage{verbatim}
\usepackage{lipsum}
\usepackage{multicol}
\usepackage{tikz}
\usetikzlibrary{arrows}
\usepackage{mathrsfs}
\usetikzlibrary{arrows}
\usepackage{listings}
\usetikzlibrary{patterns}
\definecolor{DimGray}{rgb}{0.41, 0.41, 0.41}
\usepackage{float}
\usepackage{bm}
\usepackage{mathptmx}
\usepackage[pdf]{graphviz}
\definecolor{zzttqq}{rgb}{0.6,0.2,0.}
\definecolor{uuuuuu}{rgb}{0.26666666666666666,0.26666666666666666,0.26666666666666666}
\definecolor{ududff}{rgb}{0.30196078431372547,0.30196078431372547,1.}
\definecolor{xdxdff}{rgb}{0.49019607843137253,0.49019607843137253,1.}
\theoremstyle{plain}
\newtheorem{theorem}{Theorem}[section]
\newtheorem{proposition}[theorem]{Proposition}
\newtheorem*{theorem*}{Theorem}
\newtheorem{lemma}[theorem]{Lemma}
\newtheorem{corollary}[theorem]{Corollary}

\theoremstyle{definition}
\newtheorem{defin}[theorem]{Definition}

\newtheorem{ex}[theorem]{Example}

\theoremstyle{remark}
\newtheorem{rem}[theorem]{Remark}
\newcommand{\ord}{\operatorname{ord}}

\title[Value set of modules]{Curve singularities with one Puiseux pair and value sets of modules over their local rings}
\author[M. Alberich-Carrami\~{n}ana]{Maria Alberich-Carrami\~{n}ana}
\author[P. Almir\'on]{Patricio Almir\'on}
\author[J.J. Moyano-Fern\'andez]{Julio-Jos\'e Moyano-Fern\'andez}



\thanks{The first author is supported by the grant PID2019-103849GB-I00 funded by MCIN/ AEI /10.13039/501100011033. The second author is supported by Spanish Ministerio de Ciencia, Innovaci\'{o}n y Universidades PID2020-114750GB-C32 and by the IMAG–Maria de Maeztu grant CEX2020-001105-M / AEI /
	10.13039/501100011033, through a postdoctoral contract. The third author was partially supported by
	MCIN/AEI/10.13039/501100011033 and by "ERDF -- A way of making Europe", grants PGC2018-096446-B-C22 and RED2018-102583-T, as well as by Universitat Jaume I, grant UJI-B2021-02.}

\address{Institut de Rob\`otica i Inform\`atica Industrial (IRI, CSIC-UPC)\\
Departament de Matemàtiques \& Institut de Matemàtiques de la UPC-BarcelonaTech (IMTech), Universitat Politècnica de Catalunya · {BarcelonaTech}, 
Av. Diagonal 647, 08028 Barcelona, Spain.}

\email{maria.alberich@upc.edu}

\address{Instituto Universitario de Matemáticas (IMAG), Universidad de Granada,
C/ Ventanilla, 11, 18071, Granada, Spain.}

\email{patricioalmiron@ugr.es}

\address{Universitat Jaume I, Campus de Riu Sec, Departamento de Matem\'aticas \& Institut Universitari de Matem\`atiques i Aplicacions de Castell\'o, 12071
Caste\-ll\'on de la Plana, Spain.}

\email{moyano@uji.es}

\begin{document}

\begin{abstract}
In this paper we characterize the value set $\Delta$ of the \(R\)-modules of the form \(R+zR\) for the local ring \(R\) associated to a germ $\xi$ of an irreducible plane curve singularity with one Puiseux pair. In the particular case of the module of K\"ahler differentials attached to $\xi$, we recover some results of Delorme. From our characterization of $\Delta$ we introduce a proper subset of semimodules over the value semigroup of the ring \(R\). Moreover, we provide a geometric algorithm to construct all possible semimodules in this subset for a given value semigroup.
\end{abstract}

\maketitle

\section{Introduction}

Let \((C,0)\) be an irreducible plane curve singularity  with any number of Puiseux pairs defined by a germ of function \(f.\) Associated to the local ring \(R:=\mathbb{C}\{x,y\}/(f)\) of \((C,0),\) one has naturally a discrete valuation \(v:R\rightarrow\mathbb{N}\cup\{\infty\}.\) Moreover, the value set of the local ring of the curve \mbox{\(\Gamma(C):=v(R)=\langle \overline{\beta}_0,\dots,\overline{\beta}_g\rangle\)} has the structure of numerical semigroup, and ---more remarkably--- it is known to be a topological invariant of the singularity \cite{zariski}. If we set \(n_i:=\gcd(\overline{\beta}_0,\dots,\overline{\beta}_{i-1})/\gcd(\overline{\beta}_0,\dots,\overline{\beta}_{i})\), then \(\Gamma(C)\) satisfies the following properties
\begin{enumerate}
	\item \label{eq:semigroupbranch1bisbis} \(n_i\overline{\beta}_i\in\langle \overline{\beta}_0, \overline{\beta}_1, \dots, \overline{\beta}_{i-1} \rangle\),
	\item \label{eq:semigroupbranch2bisbis} \(n_i\overline{\beta}_i<\overline{\beta}_{i+1}\) for all \(i=1,\dots,g.\)
\end{enumerate}
\medskip

In 1972, Bresinsky \cite[Theorem 2]{bresinsky} and Teissier \cite[Chap. I.~3.2]{teissier-appendix} independently proved that, for any numerical semigroup \(\Gamma\) satisfying conditions \eqref{eq:semigroupbranch1bisbis} and \eqref{eq:semigroupbranch2bisbis}, there exists a plane branch \mbox{\((C,\textbf{0})\subset(\mathbb{C}^2,\textbf{0})\)} such that \(\Gamma=\Gamma(C)\); we will refer to this statement as the Bresinsky-Teissier Theorem. The analytic counterpart of \(\Gamma\) is the value set of K\"{a}hler differentials. The value set of K\"{a}hler differentials of an irreducible plane curve singularity is defined as \(\Delta'=v(Rdx+Rdy)\); this is an analytic invariant, as proved by Delorme in 1978 \cite{delorme78}. Moreover, one can easily check that \(\Delta'\) has structure of \(\Gamma\)--semimodule, its normalization is given by \(\Delta=\Delta'-(\overline{\beta}_0-1),\) where \(\overline{\beta}_0\) is the multiplicity of \(C\) at the singular point, and $\Delta=v(R+Rdy/dx)$; therefore we will employ the terminology \emph{K\"{a}hler semimodule}.
\medskip

In this paper, we will focus on the case that $C:f(x,y)=0$ is the germ of an irreducible plane curve singularity with one Puiseux pair, i.e. a germ of plane curve with equation 
$$
f(x,y)=x^{p}-y^{q}+\sum_{iq+jp>pq}a_{i,j}x^{i}y^{j},
$$
where $\gcd(p,q)=1$. In this particular case the semigroup of values of \(C\) is minimally generated by $p$ and $q$, i.e. $\Gamma=\langle p,q\rangle:=\{n\in\mathbb{N}|\;n=ap+bq\}$. 
\medskip

The main goal of this paper is to provide an extension of the Bresinsky-Teissier Theorem to \(\Gamma\)--semimodules, which will lead to the notion of \emph{increasing semimodule}.  First of all, we provide ---with the help of the lattice path representation of the elements in $\mathbb{N}\setminus \Gamma$ due to the third author and Uliczka \cite{MU1, MU2}--- a \emph{constructive} algorithm to compute all possible increasing semimodules for a fixed \(\Gamma\) (see Section \ref{subsec:constmaltes}).  As a consequence, we prove that the set of all possible increasing semimodules of a given \(\Gamma\) has a natural tree structure (see Section \ref{subsec:constmaltes}).
After that, in a purely abstract way, our main Theorem (Theorem \ref{thm:existence-curve}) reads as follows: given both a numerical semigroup \(\Gamma=\langle p,q\rangle\) with two minimal generators and an increasing \(\Gamma\)-semimodule \(\Delta\), we can construct a Puiseux parameterization of a plane curve with value semigroup \(\Gamma\) whose local ring has an \(R\)-module \(M\) such that \(\Delta_{M}=\Delta\). In particular, the existence of an increasing semimodule induces naturally a deformation of the parameterization of the curve, see Remark \ref{rem:deformation}. This contrasts with the fact that ---in the one Puiseux pair case--- none of the results of Piontkowski \cite[Section 3]{piont} put specific conditions over the \(\Gamma\)--semimodule \(\Delta;\) therefore, they do not provide a complete classification of value set of modules. From this point of view, Theorem \ref{thm:existence-curve} can be thought as a sort of analogue to Teissier's result \cite[Chapter I,~Prop. 3.2.1]{teissier-appendix} and hence an improvement of Piontkowski's results for the case of modules over the local ring associated to a branch with one Puiseux pair. These results can be considered as the main contributions of this work.
\medskip

 In this generality, the converse of our Theorem \ref{thm:existence-curve} is still a work in progress. However, thanks to Delorme's results \cite[Lemma 12]{delorme78} we can prove it for the particular set of increasing semimodules whose first generator is \(q-p,\) assuming \(p<q\) (Corollary \ref{cor:kahlerall}). In particular, this implies that the value sets of K\"ahler differentials can be computed from the combinatorics of the increasing semimodules over $\Gamma$. We illustrate this fact with a final example obtained from the implementation of our constructive proof of Theorem \ref{thm:existence-curve} in \textsf{Mathematica} code \cite{mathe} available in the PhD Thesis of the second author \cite[Appendix A]{almtesis}.
 \medskip
 
 In the last stage of the preparation of this manuscript, different works which investigate topics of this paper have appeared, as we sum up here. In \cite{hernandesabreu}, de Abreu and Hernandes present a procedure to recover the semigroup of values of an irreducible plane curve singularity with any number of Puiseux pairs from the semimodule of K\"{a}hler differentials, which provides a test to decide whether a given subset is a semimodule of Kähler differentials or not; combining that result together with \cite{hefez1,hefez2}, they present an algorithm to obtain all possible K\"{a}hler semimodules associated to an equisingularity type based on the computation of all possible standard basis for the Tjurina algebra. As already said, our procedure has the advantage of avoiding the standard basis approach for branches with one Puiseux pair. In \cite{CCS}, Cano, Corral and Senovilla-Sanz investigate the K\"ahler differentials of a branch with one Puiseux pair using foliation theory; they consider our notion of increasing semimodule in order to study the particular case of K\"ahler semimodules, and describe the different truncations as analytic semiroots. In contrast to all these cases,  we present a more general framework which deals with modules over the ring of a plane branch with one Puiseux pair. In the particular case of K\"ahler differentials, our techniques allow to prove that all combinatorial possibilities of a semimodule to be a K\"ahler semimodule can be realizable; as far as the authors' knowledge, this fact cannot be obtained by any method in the existing literature.

\section{Numerical semigroups and semimodules}\label{sec:numerical semigroups}

A numerical semigroup $\Gamma$ is an additive sub-semigroup of the monoid $(\mathbb{N},+)$ such that the greatest common divisor of all its elements is equal to $1$. The complement $\mathbb{N}\setminus \Gamma$ is therefore finite, and its elements are called the gaps of $\Gamma$. Thus, \(\Gamma\) is finitely generated. The number $c(\Gamma)=\max (\mathbb{N}\setminus \Gamma)+1$ is called the conductor of $\Gamma$.
\medskip

A $\Gamma$-semimodule is a non-empty subset \(\Delta\) of $\mathbb{N}$ such that $\Delta+\Gamma\subseteq \Delta$. A system of generators of $\Delta$ is a subset $\mathcal{E}$ of $\Delta$ with $\Delta=\bigcup_{x\in \mathcal{E}} (x+\Gamma)$; it is called minimal if no proper subset of $\mathcal{E}$ generates $\Delta$. Notice that, since $\Delta\setminus \Gamma$ is finite, every $\Gamma$-semimodule is finitely generated, and so \(\Delta\) has a conductor, which is defined as $c(\Delta)=\max (\mathbb{N}\setminus \Delta)+1$.  Moreover, every $\Gamma$-semimodule $\Delta$ has a unique minimal system of generators (see e.g. \cite[Lemma 2.1]{MU1}). Two $\Gamma$-semimodules $\Delta$ and $\Delta'$ are called isomorphic if there is an integer $n$ such that $x\mapsto x+n$ is a bijection from $\Delta$ to $\Delta'$; we write then $\Delta\cong \Delta'$.
\medskip

For every $\Gamma$-semimodule $\Delta$ there is a unique semimodule $\Delta' \cong \Delta$ containing $0$; this semimodule is called normalized. Accordingly, the $\Gamma$-semimodule
$$
\Delta^{\circ} := \{x-\min \Delta : x \in \Delta\}
$$
is called the normalization of $\Delta$; the normalization of $\Delta$ is the unique $\Gamma$-semimodule isomorphic to $\Delta$ which contains $0$. 
\medskip

The minimal system of generators $x_0=0, x_1,\ldots , x_s$ of a normalized $\Gamma$-semimodule is what the third author and Uliczka called a $\Gamma$-lean set \cite{MU1}, i.e. it satisfies that
$$
|x_i-x_j| \notin \Gamma \ \ \mbox{for~any} \ \ 0\leq i <j \leq s,
$$
and conversely, every $\Gamma$-lean set of $\mathbb{N}$ minimally generates a normalized $\Gamma$-semimodule. Hence there is a bijection between the set of isomorphism classes of $\Gamma$-semimodules and the set of $\Gamma$-lean sets of $\mathbb{N}$, cf.~\cite[Corollary 2.3]{MU1}.
\medskip

For specific material about numerical semigroups, the reader is referred to the books of Rosales and Garc\'ia S\'anchez \cite{RosalesGarciaSanchez} and Ram\'irez Alfons\'in \cite{RamirezAlfonsin}.

\subsection{Lattices paths} \label{subs:lattice} In this paper we will restrict our attention to numerical semigroups with two generators, say $\Gamma=\langle \alpha, \beta \rangle=\mathbb{N}\alpha + \mathbb{N}\beta$ for two integer numbers $\alpha,\beta$  with $1<\alpha < \beta $ and $\mathrm{gcd}(\alpha, \beta) = 1$. In this case the conductor of $\Gamma$ is simply $c=c(\langle \alpha,\beta \rangle)=(\alpha-1)(\beta-1)$. The gaps of $\langle \alpha, \beta \rangle$ are also easy to describe: they are of the form $\alpha \beta -a\alpha -b \beta$, where $a\in\ ]0,\beta-1]$ and $b\in \ ]0,\alpha-1]$. This description yields a map from the set of gaps of $\langle \alpha, \beta \rangle$ to $\mathbb{N}^2$ given by $\alpha \beta -a\alpha -b \beta \mapsto (a,b)$, which allows us to identify a gap with a lattice point; since the gaps are positive numbers, the point lies inside the triangle $\mathcal{T}_{\alpha,\beta}$ with vertices $(0,0),(0,\alpha),(\beta, 0)$. 
\medskip

In the following, we will use the notation
$$
x_i=\alpha\beta - a_i\alpha-b_i\beta \ \   \ \mbox{or} \ \ \ x_i=\alpha\beta - a(x_i) \alpha-b(x_i)\beta
$$
for a gap $x_i$ of the semigroup $\langle \alpha,\beta \rangle$. We define a partial ordering $\preceq$ on the set of gaps as follows:

\begin{defin}\label{def:order}
	For gaps $x_1,x_2$ of $\langle \alpha , \beta \rangle$, we define 
	$$
	x_1 \preceq x_2 \ \ :\Longleftrightarrow \ a_1\leq a_2 \ \ \wedge \ \ b_1 \geq b_2
	$$
	and
	$$
	x_1 \prec x_2 \ \  :\Longleftrightarrow  \ a_1< a_2 \ \ \wedge \ \ b_1 >b_2.
	$$
\end{defin}

Let $\mathcal{E}=\{0,x_1,\ldots , x_s\} \subseteq \mathbb{N}\setminus\Gamma$ be a subset of gaps of $\Gamma=\langle \alpha, \beta  \rangle$ such that for every $i=1,\ldots, s$ it fulfills $a_1<a_2<\cdots < a_s$. Corollary 3.3 in \cite{MU1} ensures that $\mathcal{E}$ is $\langle \alpha , \beta \rangle$-lean if and only if $b_1>b_2>\cdots > b_s$. This simple fact leads to an identification between an $\langle \alpha, \beta \rangle$-lean set and a lattice path with steps downwards and to the right from $(0,\alpha)$ to $(\beta,0)$ not crossing the line joining these two points, where the lattice points identified with the gaps in $\mathcal{E}$ mark the turns from the $x$-direction to the $y$-direction, see \cite[Lemma 3.4]{MU1}; these turns will be called ES-turns for abbreviation. As an example, Figure \ref{fig:lattice1} shows the lattice path corresponding to the $\langle 5,7 \rangle$-lean set $[0,9,6,8]$.

\begin{center}
	\begin{figure}[H]
		\begin{tikzpicture}[scale=0.75]
		\draw[] (0,0) grid [step=1cm](7,5);
		\draw[] (0,5) -- (7,0);
		\draw[ultra thick] (0,5) -- (0,3) -- (1,3) -- (1,2) -- (3,2) -- (3,1) -- (4,1) -- (4,0) -- (7,0);
		\draw[fill=white] (0,5) circle [radius=0.1]; 
		\draw[fill] (1,3) circle [radius=0.1]; 
		\draw[fill] (3,2) circle [radius=0.1]; 
		\draw[fill] (4,1) circle [radius=0.1]; 
		\draw[fill=white] (7,0) circle [radius=0.1]; 
		
		\node [below right][DimGray] at (0.15,0.8) {$23$};
		\node [below right][DimGray] at (1.15,0.8) {$18$};
		\node [below right][DimGray] at (2.15,0.8) {$13$};
		\node [below right] at (3.25,0.8) {$8$};
		\node [below right][DimGray] at (4.25,0.8) {$3$};
		
		\node [below right][DimGray] at (0.15,1.8) {$16$};
		\node [below right][DimGray] at (1.15,1.8) {$11$};
		\node [below right] at (2.15,1.8) {$6$};
		\node [below right][DimGray] at (3.25,1.8) {$1$};
		
		\node [below right] at (0.25,2.8) {$9$};
		\node [below right][DimGray] at (1.25,2.8) {$4$};
		
		\node [below right][DimGray] at (0.25,3.8) {$2$};
		\end{tikzpicture}

		\caption{Lattice path for the $\langle 5,7 \rangle$-lean set $[0,9,6,8]$.} \label{fig:lattice1}
	\end{figure}
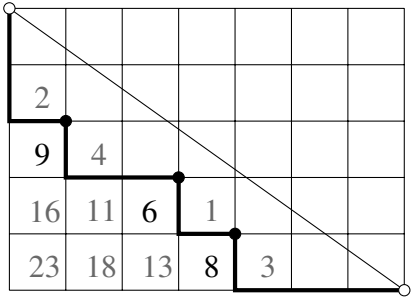
\end{center}

We will use the notation $[0,x_1,\ldots , x_s]$ for the minimal system of generators of a $\Gamma$-semimodule just to emphasize the ordering of Definition \ref{def:order}.
\medskip

Let $[g_0=0,g_1,\ldots , g_s]$ be the minimal system of generators of a $\langle \alpha, \beta \rangle$-semimodule $\Delta$. In \cite{MU1} the third author and Uliczka introduced the notion of syzygy of $\Delta$ as the $\langle \alpha, \beta \rangle$-semimodule
$$
\mathrm{Syz}(\Delta):=\bigcup_{{\tiny\begin{array}{c}
i,j\in \{0,\ldots , s\}\\
i\neq j\end{array}}} \Big ((\Gamma + g_i)\cap (\Gamma + g_j) \Big ).
$$
According to \cite[Theorem 4.2]{MU1}, for $\Gamma=\langle \alpha, \beta \rangle$ the syzygies of a normalized $\Gamma$-semimodule \(\Delta\) can be characterized as follows:

\begin{proposition}\cite[Theorem 4.2]{MU1}\label{defin:syz}
	Let \(\Delta\) be a \(\Gamma\)-semimodule with minimal system of generators $[g_0=0,g_1,\dots,g_s]$. Assume that the minimal system of generators is ordered with the gap-order, i.e. \(g_0\prec g_1\prec\cdots\prec g_s.\) Then the syzygy of $\Delta$ is the set
	\begin{equation*}
	\mathrm{Syz}(\Delta) =\bigcup_{0\leq k<j\leq s} \Big ((\Gamma + g_k)\cap (\Gamma + g_j) \Big )= \bigcup_{k=0}^{s} (\Gamma + h_k),
	\end{equation*}
	
	where $h_1,\ldots , h_{s-1}$ are gaps of $\Gamma$, $h_0,h_s \leq \alpha \beta$, and
	
	\begin{align*}
	&h_k \equiv g_k ~\mathrm{ mod }~ \alpha, ~h_ k > g_{k}~~  \ \mbox{for }~~ k=0,\ldots , s \\
	&h_k \equiv g_{k+1} ~ \mathrm{mod } ~\beta,~ h_ k > g_{k+1}~~ \ \mbox{for } ~~k=0,\ldots , s-1\\
	&h_s \equiv 0~\mathrm{mod } ~\beta, \ \mathrm{and } ~h_s > 0.
	\end{align*}
\end{proposition} 

\begin{rem}
	Observe that the modular conditions for the generators of the syzygy module \([h_0,\dots,h_s]\) give us explicit expressions for \(h_i\) in terms of the coordinates of the minimal system of generators of \(\Delta.\) 
\medskip

	Assume that, for the minimal system of generators $[g_0=0,g_1,\dots,g_s]$  of \(\Delta\), for any $i=1,\ldots ,s$ we write $g_i=\alpha\beta-a_i\alpha-b_i\beta$ i.e. with \(g_0\prec g_1\prec\cdots\prec g_s.\) Then, 
	\[h_i=\alpha\beta-a_{i-1}\alpha-b_i\beta.\]
\end{rem}

\begin{ex}
Let us consider again $\Gamma=\langle 5,7\rangle$ and the semimodule $\Delta$ with minimal system of generators $[0,9,6,8]$ and lattice path as in Figure \ref{fig:lattice1}. Then the $\langle 5,7\rangle$-semimodule \(\mathrm{Syz}(\Delta)\) is minimally generated by $h_0=15$, $h_1=13$, $h_2=16$ and $h_3=14$.
\end{ex}

The syzygies allowed the second and third authors \cite{almiyano} to give a formula for the conductor $c(\Delta)$ of $\Delta$:

\begin{theorem}\cite[Theorem 1]{almiyano}\label{formulaconductor}
	Let \(\Delta\) be a \(\Gamma\)-semimodule. Let \(I\) be a minimal system of generators of \(\Delta\) and \(J\) be a minimal system of generators of \(\operatorname{Syz}(\Delta)\). Let \(M:=\max_{\leq_\mathbb{N}}\{h\in J\}\) denote the biggest, with respect to the usual ordering of the natural numbers, minimal generator of the syzygy module. Then
	\[c(\Delta)=M-\alpha-\beta+1.\]
	In particular, if we denote by \((m_1,m_2)\) the point in the lattice $\mathcal{L}$ representing \(M\), then we have
	\[c(\Delta)=c(\Gamma)-m_1\alpha-m_2\beta.\]
\end{theorem}
\begin{rem}
	From now on, all minimum and maxima will be taken under the order of the natural numbers; we will then write \(\max_{\leq_{\mathbb{N}}},\min_{\leq_{\mathbb{N}}}.\)
\end{rem}

As said, with the ordering \(g_0\prec g_1\prec \dots \prec g_s\) for the elements of $\Delta$ we write
\[
E_i=\bigcup_{0\leq j\leq i}(\Gamma+g_j)\;\;\;\;\text{for}\;0\leq i\leq s.
\]

For any $i=1,\dots,s$, the number $u_{i}:=\min_{\leq_{\mathbb{N}}}\{(\Gamma+g_i)\cap E_{i-1}\}$ will play an important role in the sequel. We may write $u_i$ in terms of the coordinates in the lattice path $\mathcal{L}$:

\begin{proposition}\label{prop:uisyz}
	For any $i=1,\ldots ,s$ and $g_i=\alpha\beta-a_i\alpha-b_i\beta$, let us define $a_0=b_0=0$, $x_i=\alpha(\beta-a_i)$, $y_i=\beta(\alpha-b_i)$ and $m_i=\min_{\leq_\mathbb{N}}(x_i, y_i)$. Let us assume that \(g_0\prec g_1\prec \dots \prec g_s\). Then
	$$
	u_i=\min_{\leq_\mathbb{N}} \{\alpha\beta-a_{i-1}\alpha-b_i\beta, m_i\}.
	$$
\end{proposition}

\begin{proof}
	Along the whole proof \(\min:=\min_{\leq_{\mathbb{N}}}.\) For every $i=0,\ldots, s$ and $j=1,\ldots, s$ set $\Omega_{i,j}:=(\Gamma+g_i)\cap (\Gamma+g_j)$ and $\omega_{i,j}:=\mathrm{min}\ \Omega_{i,j}$.
	It is easily checked that
	$$
	\begin{array}{lcl}
	\omega_{0,i}&=&\min \{\beta(\alpha-b_i),\alpha(\beta-a_i)\}\\
	\omega_{i,i+j} &=&\alpha\beta-a_i\alpha-b_{i+j}\beta
	\end{array}
	$$
	\begin{center}
		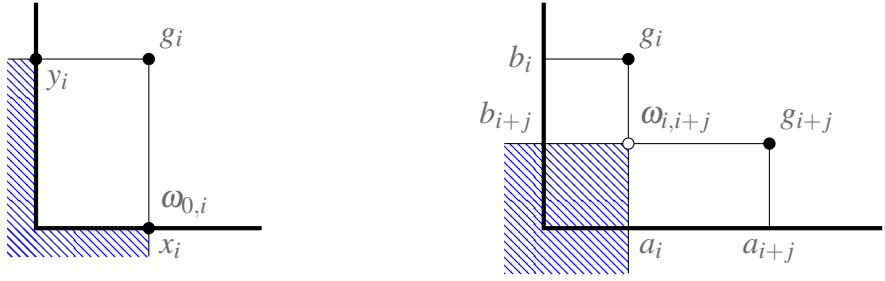
\begin{figure}[H]
			\begin{tikzpicture}[scale=0.75]
			
			\draw[draw=none, pattern=north west lines, pattern color=blue] (8.3,-0.8) rectangle (10.5,1.5);
			\draw[draw=none, pattern=north west lines, pattern color=blue] (-0.5,-0.5) rectangle (0,3);
			\draw[draw=none, pattern=north west lines, pattern color=blue] (-0.5,-0.5) rectangle (2,0);
   			\draw[ultra thick] (0,4) -- (0,0) -- (4,0);
			\draw[fill] (2,3) circle [radius=0.1]; 
			\draw[] (2,3) -- (2,-0.5);
			\draw[] (2,3) -- (-0.5,3);
			\draw[fill] (0,3) circle [radius=0.1]; 
			\draw[fill] (2,0) circle [radius=0.1]; 
			
			\node [below right][DimGray] at (0,3) {$y_i$};
			\node [below right][DimGray] at (2,0) {$x_i$};
			\node [above right][DimGray] at (2,0) {$\omega_{0,i}$};
			\node [above right][DimGray] at (2,3) {$g_i$};
			
			\draw[ultra thick] (9,4) -- (9,0) -- (15,0);
			\draw[fill] (10.5,3) circle [radius=0.1]; 
			\draw[] (10.5,3) -- (10.5,-0.8);
			\draw[] (10.5,3) -- (9,3);
			\node [below right][DimGray] at (10.5,0) {$a_i$};
			\node [below ][DimGray] at (13,0) {$a_{i+j}$};
			\node [above left][DimGray] at (9,1.5) {$b_{i+j}$};
			\node [left][DimGray] at (9,3) {$b_{i}$};
			
			\draw[fill] (13,1.5) circle [radius=0.1]; 
			\draw[] (13,1.5) -- (13,0);
			\draw[] (13,1.5) -- (8.3,1.5);
			
			\draw[fill=white] (10.5,1.5) circle [radius=0.1]; 
			\node [above right][DimGray] at (10.5,1.5) {$\omega_{i,i+j}$};
			\node [above right][DimGray] at (13,1.5) {$g_{i+j}$};
			\node [above right][DimGray] at (10.5,3) {$g_i$};
			\end{tikzpicture}	
			
			\caption{Shaded regions $(\Gamma + x_i)\cup (\Gamma+y_i)$ and $\Omega_{i,i+j}$.} \label{fig:lattice2}
		\end{figure}
	\end{center}

	Now for $i=1,\ldots, s$ we get
	\begin{align*}
	u_i=&\min \{(\Gamma+g_i)\cap E_{i-1}\}=\min \Big \{(\Gamma+g_i)\cap \bigcup_{j=0}^{j-1}(\Gamma+g_j)\Big \}\\
	=& \min (\Omega_{0,i}\cup \Omega_{1,i}\cup \cdots \cup \Omega_{i-1,i})=\min \{\omega_{0,i}, \omega_{1,i},\ldots , \omega_{i-1,i}\}.
	\end{align*}
	Observe that $\omega_{j,i}<\omega_{j+1,i}$ for $j=1,\ldots , i-1$ since $a_{k}>a_{k+1}$ for all $k=1,\ldots, i-2$; therefore $u_i=\min\{\omega_{0,i},\omega_{i-1,i}\}$, as desired.
\end{proof}

\begin{rem}
	If all \(u_i\) became of the form \(\alpha\beta-a_{i-1}\alpha-b_i\beta\), then \(u_i=h_i,\) i.e. they are minimal generators of the semimodule of syzygies.
\end{rem}

Finally, we recall a few properties of the \(u_i\)'s previously defined which will be very useful in the sequel. These properties were already given by Delorme {delorme78} in a different context:
\begin{lemma}\cite[Lemma~10]{delorme78} \label{lem:1}
Let $p,q\in\mathbb{Z}$ be such that $|p-q|\notin\Gamma$. We set
	$$
	u:=\min_{\leq_{\mathbb{N}}}\{(\Gamma+p)\cap(\Gamma+q)\}
	$$
	as well as $\bar{u}:=u+c(\Gamma)-\alpha\beta$, $v:=p+q+\alpha\beta-u$ and $\bar{v}:=v+c(\Gamma)-\alpha\beta$. Then we have:
	\begin{enumerate}
		\item $(\Gamma+p)\cap(\Gamma+q)=(\Gamma+u)\cup(\Gamma+v)$,
		\item $(\Gamma+p)\cup(\Gamma+q)=(\Gamma+u-\alpha\beta)\cap(\Gamma+v-\alpha\beta)$,
		\item $\mathbb{N}+\bar{v}\subset(\Gamma+p)\cup(\Gamma+q)$,
		\item $(\mathbb{N}+\bar{u})\cap((\Gamma+p)\cup(\Gamma+q))=(\mathbb{N}+\bar{u})\cap(\Gamma+v-\alpha\beta)$.
	\end{enumerate}
\end{lemma}




\subsection{Increasing semimodules}

From now on, and throughout the paper, when referring to the minimal set of generators \(\{g_0,g_1,\dots,g_s\}\) of a \(\Gamma\)--semimodule \(\Delta,\) we will assume that it is ordered by the natural ordering \(g_0<_{\mathbb{N}} g_1<_{\mathbb{N}}\cdots<_{\mathbb{N}}g_s\) unless we mention the contrary. For the scope of this paper, we are interested in a particular subset of semimodules:

\begin{defin}\label{defin:increasingsemimod}
	Let \(\Gamma\) be a numerical semigroup minimally generated by \(\langle\alpha,\beta\rangle.\) A \(\Gamma\)--semimodule $L$ is called an \textit{increasing semimodule} if it satisfies the following property:
	\begin{equation}\label{maltese}\tag{\(\maltese\)}
		\begin{split}
			&\text{If $L$ has minimal set of generators} \;\{g_0=0,g_1,\dots,g_s\}\\
			&\text{and we put}\;g_{s+1}=\infty,\;u_0=0,\\
			&\text{then for all $0\leq i\leq s$ we have} \;\;g_{i+1}>u_i,\\
			&\text{where $u_i=\min\{(\Gamma+g_i)\cap E_{i-1}\}$}\;\text{for $1\leq i\leq s$} \\
			& \text{and 
				$E_i=\bigcup_{0\leq j\leq i} (\Gamma+g_j)$ for $0\leq i\leq s$.}
		\end{split}
	\end{equation}
\end{defin}

\begin{rem}
	Our definition is in the setting of a numerical semigroup minimally generated by two elements since, the value sets of K\"{a}hler differentials for curve singularities with more than one Puiseux pair are not in general increasing semimodules: let us consider an irreducible plane curve singularity with semigroup \(\langle 4,6,17\rangle\) with equation 
	\[f(x,y)=(y^2-x^3)^2-x^7y.\]
	With the aid of \textsf{Singular} \cite{singular}, we can calculate the minimal generators of the module of K\"{a}hler differentials by computing a standard basis of the ideal \((f,\partial f/\partial x,\partial f/\partial y).\) Once we have obtained the standard basis, it is easy to check that the normalized set of values is minimally generated by \(\{g_0=0,g_1=2,g_2=11,g_3=13\}.\) A straightforward computation shows that \hbox{ \(u_2=\min\{E_1\cap \Gamma+g_2\}=17>g_3.\)} However, it would be certainly interesting to explore the properties of the increasing semimodules defined over a numerical semigroup with arbitrarily many minimal generators.
\end{rem}

Before continuing, let make us clear that the class of increasing semimodules is non-empty:
\begin{lemma}\label{lem:semitwogen}
	Any normalized semimodule with two generators is increasing.
\end{lemma}

\begin{proof}
	Because of the minimal set of generators is of the form $[0,g]$ with $g\in\mathbb{N}\setminus\Gamma$, then condition \eqref{maltese} trivially holds since $g_2=\infty$ and 
	$u_0=0$
\end{proof}

Also, an immediate consequence of Delorme's Theorem \ref{thm:valuesemimod} is the following 
\begin{corollary}
	Let be \(\Delta=v(R+Rdy/dx)\), where \(R\) is the local ring of an irreducible plane curve singularity with one Puiseux pair \((\alpha,\beta)\). Then \(\Delta\) is an increasing \(\Gamma\)--semimodule, where \(\Gamma=\langle\alpha,\beta\rangle.\)
\end{corollary}
\begin{proof}
	By Theorem \ref{thm:valuesemimod}, the generators \(w_i\) of \(R+Rdy/dx\) satisfy \mbox{\(\displaystyle \omega_{i+1}=\sum_{0\leq j\leq i}F_{j,i}\omega_{j}\)} where \(F_{j,i}\in R,\)  \(u_i=v(F_{i,i})+g_i=\inf_j\{v(F_{j,i})+g_j\}\) and \(v(\omega_{i+1})=g_{i+1}.\) Since 
	$$
	v(w_{i+1})\geq \min\{v(F_{j,i})+v(w_j),
	$$
 the claim follows.
\end{proof}

Moreover, the set of increasing semimodules is a \emph{proper} non-empty subset in the set of \(\Gamma\)-semimodules: consider for instance the \(\Gamma\)--semimodule minimally generated by \(\{0,6,8,9\}\) (see Figure \ref{fig:lattice1}); this is not an increasing semimodule since \(\inf(\langle 5,7\rangle \cap(\langle 5,7\rangle+6))=20>8\). 
\medskip

To finish with this section, we prove a rather technical lemma which will be useful later.
\begin{lemma}\label{lem:3}
	\label{cslemma} Let be $\Gamma=\langle \alpha,\beta\rangle$. Let $L$ be an increasing \(\Gamma\)-semimodule with $\{g_0=0,g_1,\dots,g_s\}$, we set \(g_{s+1}=\infty\) and \(\bar{u}_i:=u_i+c(\Gamma)-\alpha\beta\). Then, for any $i=0,\dots,s-1$
	there exists an element \(c_i\in (-\Gamma)\), namely \(c_i=c_{i-1}+g_i-u_i\), such that
	$$
	(\mathbb{N}+\bar{u}_i)\cap E_i=(\mathbb{N}+\bar{u}_i)\cap(\Gamma+c_i).
	$$
\end{lemma}

\begin{proof}
	We proceed by induction on $s$. The case $s=1$ is easily deduced from Lemma \ref{lem:1} with $p=g_0=0$, $q=g_1$; then $u_1=\min\{(\Gamma+g_1)\cap \Gamma\}\in \Gamma$, therefore it is enough to consider
	\begin{align*}
		c_1&=v-\alpha \beta =g_0+g_1+\alpha\beta-u_1-\alpha \beta \\
		& =g_1-u_1\in (-\Gamma).
	\end{align*}
	Since $\overline{u}_{s}<u_{s}<g_{s+1}$, we have $g_{s+1}\in \mathbb{N}+\overline{u}_{s}$; this together with the fact that $g_{s+1}\notin E_{s}$ implies that $g_{s+1}\notin \mathbb{N}+c_{s}$ by induction hypothesis. Therefore $g_{s+1}-c_{s}\notin \Gamma$ and we apply Lemma \ref{lem:1} (4) with $p=g_{s+1}$ and $q=c_{s}$, so that $u_{s+1}=\min\{(\Gamma+g_{s+1})\cap E_{s}\}$. If we set $c_{s+1}=g_{s+1}-u_{s+1}+c_{s}$, then Lemma \ref{lem:1} again yields the equality
	$$
	(\mathbb{N}+\overline{u}_{s+1})\cap E_{s+1} = (\mathbb{N}+\overline{u}_{s+1}) \cap (\Gamma+c_{s+1}),
	$$
	as desired. 
\end{proof}
\begin{rem}
	Lemma \ref{lem:3} was proven by Delorme \cite[Lemma 12 (a)]{delorme78} for the case of an increasing semimodule with \(g_1=\beta-\alpha.\)
\end{rem}

\begin{rem}
	By the proof of Lemma \ref{lem:3} itself, we observe that the statement of that result still holds if we substitute the assumption \(g_{i+1}>u_i\) for all \(i\) by the hypothesis $g_{i+1}>\overline{u}_{i}$ for all \(i\).
\end{rem}

\subsection{Lattice paths of increasing semimodules}\label{subsec:constmaltes}

To conclude the section, we are going to show a procedure to construct any lattice path associated to an increasing semimodule. Recall that we are assuming that the minimal system of generators \(g_0=0, g_1, \ldots, g_s\) of a \(\Gamma\)--semimodule is ordered by the natural ordering  \(g_0=0<_\mathbb{N} g_1<_\mathbb{N}\cdots <_\mathbb{N} g_s\).
\medskip

An easy consequence of Lemma \ref{lem:semitwogen} is that any lattice path with a unique ES-turn is an increasing semimodule \(\Delta^{(1)}\) generated by \(g_0=0,\,g_1\in\mathbb{N}\setminus\Gamma\). Let us write \(J=\{h_0,h_1\}\) for the minimal set of generators of the syzygy semimodule $\mathrm{Syz}(\Delta^{(1)})$. Therefore, it is a straightforward computation to check that \(\mathrm{min}\{h_0,h_1\}=\mathrm{min}\{\Gamma\cap (\Gamma+g_1)\}\). 
\medskip

Write \(u_1=\mathrm{min}\{h_0,h_1\}\). There exists an increasing semimodule \(\Delta^{(2)}\) with three generators containing \(\Delta^{(1)}\) if and only if there is an element \(g_2\in\mathbb{N}\setminus\Delta^{(1)}\) with \(g_2> u_1\) such that 
$$
\Delta^{(2)}=\Delta^{(1)}\cup(\Gamma+g_2). 
$$

Since \(u_1\) is a generator of the syzygy module and \(g_2\) must be a gap, then \(g_2>u_1\) means 
$$
g_2\in\{(a,b)\in\mathbb{N}^2 : \;u_1<\alpha\beta-a\alpha-b\beta\}.
$$

On the other hand, the condition \(g_2\in\mathbb{N}\setminus\Delta^{(1)}\) means that \(g_2\) is a point above the lattice path associated to \(\Delta^{(1)}\). Let us denote by \(L(\Delta^{(1)})^{+}\) the region above the lattice path associated to \(\Delta^{(1)}\). Hence the existence of \(\Delta^{(2)}\) is equivalent to 
$$
L(\Delta^{(1)})^{+}\cap\{(a,b)\in\mathbb{N}^2: \;u_1<\alpha\beta-a\alpha-b\beta\}\neq \emptyset.
$$

So, let us assume that we start with \(\Delta^{(i-1)}\) minimally generated by \(I=\{g_0=0,g_1,\dots,g_{i-1}\}\).  Consider \(u_{i-1}=\mathrm{min}\{(\Gamma+g_{i-1})\cap E_{i-2}\}\). Observe that, since the \(u_i\)'s coincide with SE-turns by construction, the syzygy module \(\mathrm{Syz}(\Delta^{(i-1)})\) is minimally generated by
$$
\{ u_1,\dots,u_{i-1},M\},
$$
where \(M=c(\Delta^{(i-1)})+\alpha+\beta-1\). We observe that we are ordering \(g_i\) here by the natural order in \(\mathbb{N}\), and this ordering does not necessarily coincide with the order \(\prec\). So, the indices in the minimal set of the syzygies may not coincide with those in Proposition \ref{defin:syz}. 
\medskip

As before, let us denote \(L(\Delta^{(i-1)})^{+}\) the region above the lattice path associated to \(\Delta^{(i-1)}\), then, there is a \(g_i\in\mathbb{N}\setminus \Gamma\) with \(g_i>u_{i-1}\) if and only if
$$
L(\Delta^{(i-1)})^{+}\cap\{(a,b)\in\mathbb{N}^2 :\;u_{i-1}<\alpha\beta-a\alpha-b\beta\}\neq \emptyset.
$$

The previous construction gives us the following immediate consequence:
\begin{proposition}\label{uisyz2}
	Let \(\Gamma=\langle\alpha,\beta\rangle\) and et \(\Delta\) be an increasing \(\Gamma\)--semimodule minimally generated by \hbox{\(\{g_0=0<_{\mathbb{N}}\cdots<_\mathbb{N}g_s\}.\)} Then, \(u_1<_{\mathbb{N}}\cdots<_\mathbb{N}u_s<_\mathbb{N}M:=c(\Delta)+\alpha+\beta-1\) are the minimal set of generators of \(\operatorname{Syz}(\Delta).\)
\end{proposition}

\begin{rem}
	Proposition \ref{uisyz2} shows that, different from Proposition \ref{defin:syz}, the minimal set of generators of the syzygy semimodule of an increasing semimodule can be obtained with the natural ordering. Obviously the labeling may differ from the order established in the lattice path.
\end{rem}
\begin{ex}\label{ex:220}
To see how the ordering in the labeling of the minimal generators of the syzygy semimodule may differ, let us consider \(\Gamma=\langle 7,9\rangle\) and the increasing \(\Gamma\)--semimodule \(\Delta\) minimally generated by $0, 5, 20$ and $31$. (This will appear in Example \ref{ex:constincreasing} again). Observe that \(\operatorname{Syz}(\Delta)\) is minimally generated by \(u_1=14,u_2=27,u_3=38\) and \(c(\Delta)+\alpha+\beta-1=40.\) If we order the minimal generators of the syzygy module with the order \(\prec\) then we have \(h_0=14,h_1=40,h_2=38,h_3=27.\)
\end{ex}

The previous construction encloses a rooted tree structure over the set of incresing semimodules in terms of its first non zero minimal generator; the root corresponds to the semimodule associated to a gap of the semigroup, and represents the unique \(\Gamma\)--semimodule of the form \(\Delta=[0,g]\). We assign this to the level 0 of the tree. The next level represent the possible increasing \(\Gamma\)-semimodules with three generators $[0, g_1, g_2]$ with $g_1=g$; hence the number of leaves at this level is 
$$
|L(\Delta^{(1)})^{+}\cap\{(a,b)\in\mathbb{N}^2: \;u_1<\alpha\beta-a\alpha-b\beta\}|.
$$

In general the number of nodes at a level \(k\) represent the number of increasing \(\Gamma\)-semimodules with \(g\) as first non zero generator and \(k+2\) minimal generators. To each node at level \(k\) we attach exactly
$$
|L(\Delta^{(k+1)})^{+}\cap\{(a,b)\in\mathbb{N}^2: \;u_{k+1}<\alpha\beta-a\alpha-b\beta\}|
$$
leaves. Obviously, this tree representation is finite; observe that $k\leq \alpha$. Let us show the procedure with an example:

\begin{ex}\label{ex:constincreasing}
	Let us consider the semigroup $\Gamma=\langle 7,9\rangle$. We are going to construct all possible increasing $\Gamma$-semimodules with $g_1=5$. Since \(u_1=14,\) the first step of the above procedure says that there are eight increasing $\Gamma$-semimodules with $g_1=5$ and three minimal generators, see Figure \ref{fig:ex1}.
	
	\begin{figure}[H]
		\begin{center}
			\begin{tikzpicture}[scale=0.65]
				
				\draw[] (0,0) grid [step=1cm](9,7);
				
				\fill[line width=1.6pt,color=blue,fill=blue,pattern=north east lines,pattern color=blue] (0.,5.444444444444445) -- (0.,1.) -- (5.714285714285714,1.) -- cycle;
				\draw [] (0.,7.)-- (9.,0.);
				\draw  [](0.,7.)-- (0.,1.);
				\draw [](0.,1.)-- (7.,1.);
				\draw [] (7.,1.)-- (7.,0.);
				\draw [] (7.,0.)-- (9.,0.);
				\draw [domain=-0.5:7.5] plot(\x,{(--49.-7.*\x)/9.});
				\draw [line width=1.0pt,color=black] (0.,5.444444444444445)-- (0.,1.);
				\draw [color=black] (0.,1.)-- (5.714285714285714,1.);
				\draw [] (5.714285714285714,1.)-- (0.,5.444444444444445);
				\draw[] (0,7) -- (9,0);
				
				\draw[ultra thick] (0,7) -- (0,1) -- (7,1) -- (7,0) -- (9,0);
				\draw[fill=white] (0,7) circle [radius=0.1]; 
				
				\draw[fill] (7,1) circle [radius=0.1]; 
				\draw[fill=white] (9,0) circle [radius=0.1]; 
				\draw[fill=red] (1,2) circle [radius=0.08]; 
				\draw[fill=red] (1,3) circle [radius=0.08]; 
				\draw[fill=red] (1,4) circle [radius=0.08]; 
				\draw[fill=red] (2,2) circle [radius=0.08]; 
				\draw[fill=red] (2,3) circle [radius=0.08]; 
				\draw[fill=red] (3,2) circle [radius=0.08]; 
				\draw[fill=red] (3,3) circle [radius=0.08]; 
				\draw[fill=red] (4,2) circle [radius=0.08]; 
				
			\end{tikzpicture}
			\caption{Lattice path for the $\langle 7,9 \rangle$-lean set $[0,5]$ and the eight candidates for \(g_2\) in red.}\label{fig:ex1}
		\end{center}
	\end{figure}
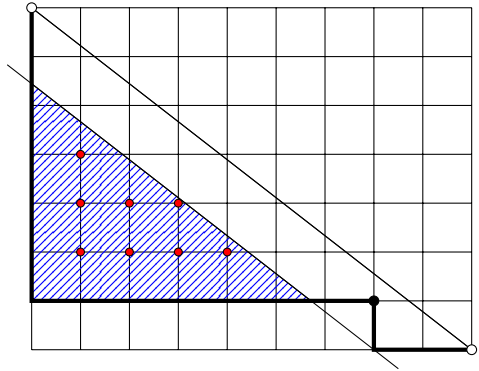

	As a second step, let us choose $g_2=20$ as minimal generator. Then, \(u_2=27\) and the following figure shows that there is only one increasing $\Gamma$-semimodule with $g_1=5,\,g_2=20$ and $4$ minimal generators in total, namely $\Delta=[0,5,20,31]$, which has \(u_3=38.\) (This semimodule has already appeared in Example \ref{ex:220}).

	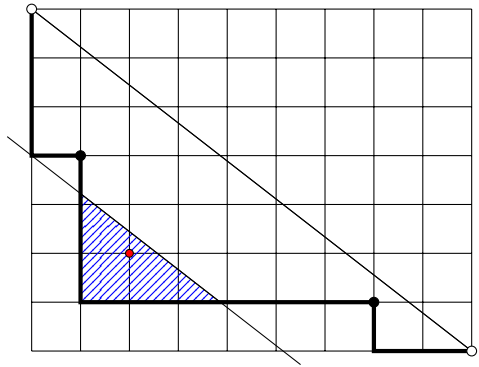
\begin{figure}[H]
		\begin{center}
			\begin{tikzpicture}[scale=0.65]
				\draw[] (0,0) grid [step=1cm](9,7);
				\draw[] (0,7) -- (9,0);
				\fill[line width=1.6pt,color=blue,fill=blue,pattern=north east lines,pattern color=blue] (1.,3.2222222222222223) -- (1.,1.) -- (3.857142857142857,1.) -- cycle;
				\draw [] (0.,7.)-- (9.,0.);
				\draw [] (7.,1.)-- (7.,0.);
				\draw [] (7.,0.)-- (9.,0.);
				\draw [] (0.,7.)-- (0.,4.);
				\draw [] (0.,4.)-- (1.,4.);
				\draw [] (1.,4.)-- (1.,1.);
				\draw [] (1.,1.)-- (7.,1.);
				\draw [domain=-0.5:5.5] plot(\x,{(--36.-7.*\x)/9.});
				\draw [color=black] (1.,3.2222222222222223)-- (1.,1.);
				\draw [color=black] (1.,1.)-- (3.857142857142857,1.);
				\draw [color=black] (3.857142857142857,1.)-- (1.,3.2222222222222223);
				\draw[ultra thick] (0,7) -- (0,4) -- (1,4) -- (1,1) -- (7,1) -- (7,0) -- (9,0);
				\draw[fill=white] (0,7) circle [radius=0.1]; 
				\draw[fill] (1,4) circle [radius=0.1]; 
				\draw[fill=red] (2,2) circle [radius=0.08]; 
				\draw[fill] (7,1) circle [radius=0.1]; 
				\draw[fill=white] (9,0) circle [radius=0.1]; 	
			\end{tikzpicture}
			\caption{Lattice path for the $\langle 7,9 \rangle$-lean set $[0,5,20]$ and candidate for \(g_3\) in red.} \label{fig:ex 2}
		\end{center}
	\end{figure}


	Finally, it is an easy computation to see that the set of increasing \(\Gamma\)-semimodules with \(g_1=5\) has the tree structure showed in Figure \ref{fig:tree}.
	
	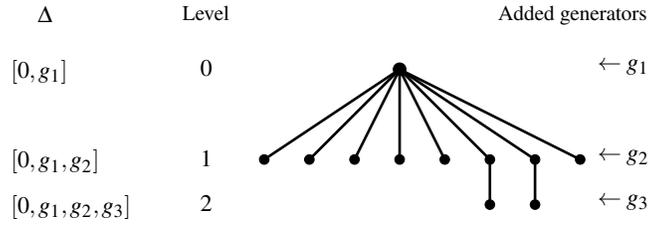
\begin{figure}[H]
		\begin{center}
			\begin{tikzpicture}[line cap=round,line join=round,>=triangle 45,x=1.0cm,y=1.0cm, scale=1.2]
				\clip(-3.,1.) rectangle (5.1,4.2);
				\draw [line width=1.pt] (2.,3.5)-- (4.,2.5);
				\draw [line width=1.pt] (2.,3.5)-- (3.5,2.5);
				\draw [line width=1.pt] (2.,3.5)-- (3.,2.5);
				\draw [line width=1.pt] (2.,3.5)-- (2.5,2.5);
				\draw [line width=1.pt] (2.,3.5)-- (2.,2.5);
				\draw [line width=1.pt] (2.,3.5)-- (1.5,2.5);
				\draw [line width=1.pt] (2.,3.5)-- (1.,2.5);
				\draw [line width=1.pt] (2.,3.5)-- (0.5,2.5);
				\draw [line width=1.pt] (3.5,2.5)-- (3.5,2.);
				\draw [line width=1.pt] (3.,2.5)-- (3.,2.);
				\begin{scriptsize}
					\draw [fill=black] (2.,3.5) circle (2.pt);
					\draw [fill=black] (4.,2.5) circle (1.5pt);
					\draw [fill=black] (3.5,2.5) circle (1.5pt);
					\draw [fill=black] (3.,2.5) circle (1.5pt);
					\draw [fill=black] (2.5,2.5) circle (1.5pt);
					\draw [fill=black] (2.,2.5) circle (1.5pt);
					\draw [fill=black] (1.5,2.5) circle (1.5pt);
					\draw [fill=black] (1.,2.5) circle (1.5pt);
					\draw [fill=black] (0.5,2.5) circle (1.5pt);
					\draw [fill=black] (3.5,2.) circle (1.5pt);
					\draw [fill=black] (3.,2.) circle (1.5pt);
					\node [below right][black] at (3,4.3) {\tiny Added generators};
					\node [below right][black] at (4.1,3.7) {$\leftarrow g_1$};
					\node [below right][black] at (4.1,2.7) {$\leftarrow g_2$};
					\node [below right][black] at (4.1,2.2) {$\leftarrow g_3$};
					
					\node [below right][black] at (-0.5,4.3) {\tiny Level};
					\node [below right][black] at (-0.3,3.7) {$0$};
					\node [below right][black] at (-0.3,2.7) {$1$};
					\node [below right][black] at (-0.3,2.2) {$2$};
					
					\node [below right][black] at (-2.1,4.3) {$\Delta$};
					\node [below right][black] at (-2.4,3.7) {$[0,g_1]$};
					\node [below right][black] at (-2.4,2.7) {$[0,g_1,g_2]$};
					\node [below right][black] at (-2.4,2.2) {$[0,g_1,g_2,g_3]$};

				\end{scriptsize}
			\end{tikzpicture}
			\vspace{-0.3cm}
			\caption{Tree of increasing \(\Gamma\)-semimodules with first non zero generator \(g=5\).}\label{fig:tree}
		\end{center}
	\end{figure}
	
\end{ex}

\section{Realization of increasing semimodules as value set of \(R\)--modules} \label{sec:R-mod}

At this point, we are ready to present the main result of this paper. We are going to show that for fixed \(\Gamma\) and an increasing \(\Gamma\)-semimodule, we can find a plane curve with value semigroup \(\Gamma\), as well as non-zero functions (elements in the local ring) such that the associated \(R\)-module has as value set the given increasing semimodule.

\begin{theorem}\label{thm:existence-curve}
	Let \(\Gamma=\langle \alpha,\beta \rangle \) be a numerical semigroup with \(\alpha<\beta\). Let \(L\) be an increasing \(\Gamma\)-semimodule, and set \(b:=c(\Gamma)-\beta-1\). Then there exist a tuple \((a_1,\dots,a_b)\in\mathbb{C}^{b}\) and \(z\in\overline{R}\) such that \(L=v(R+zR)\), where \(R\) is the local ring of the germ of plane curve singularity defined by the Puiseux parameterization
	\[C:\left\lbrace\begin{array}{ll}
		x(t):=&t^\alpha\\
		y(t):=&t^\beta+\displaystyle \sum_{i=1}^{b}a_it^{i+\beta}.
	\end{array}\right.\]	
\end{theorem}

\begin{proof}
	We first introduce some notation and definitions. Let us denote by \(g_0=0<g_1<\cdots< g_s\) the minimal set of generators of \(L\). Also set
	\begin{equation*}
		\begin{split}
			&E_i:=\bigcup_{0\leq j\leq i}(\Gamma+g_j),\ \mbox{and } \  E_s:=L\\
			&u_i:=\min\{E_{i-1}\cap (\Gamma+g_i)\}\\
			&\sigma_{i+1}:=\sum_{1\leq j\leq i}(g_{j+1}-u_j)\quad\text{with}\;\;\sigma_0=\sigma_1=0, \quad  \text{for}\;i=1,\dots,s\\
			&h_i:=g_i-\sigma_i, \quad \text{for}\;i=1,\dots,s+1  \\
			&I_i:=[\sigma_{i-1},\sigma_i]\cap\mathbb{N}, \quad \text{for}\;i=2,\dots,s+1
		\end{split}
	\end{equation*}
	
	Let us consider the polynomial ring \(\mathbb{C}[X_1,\dots,X_b,T]\). We consider formal elements \(Y,z\in \mathbb{C}[X_1,\dots,X_b,T]\) of the form
	\begin{equation*}
		Y=T^{\beta}\Big (1+\sum_{i= 1}^{b} X_iT^i\Big ),\quad z=T^{g_1}\Big (c_0+\sum_{i= 1}^{b} c_iX_iT^i \Big ).
	\end{equation*}
	
	\noindent \textbf{Step 1}: We set \(U^{0}_{\sigma_0}=1\), \(U^{1}_{\sigma_1}=z\). We will prove the existence of a family of polynomials
	$$
	\Big \{\big \{\{U^{i}_{j}\}_{j\in I_i} \big \}_{i=2,\dots,s+1}\Big \}\subset\mathbb{C}[X_1,\dots,X_b,T]
	$$
	such that
	\[
	U^{i}_{j}=\Big (X_j\prod_{k=2}^{i}A_kc_{k,j}+V_j\Big )T^{j+h_i}+\sum_{r=j+1}^{b}\Big (X_r\prod_{k=2}^{i}A_kc_{k,r}+V_r\Big )T^{r+h_i}+\mathrm{h.o.t.}
	\]
	with \(V_r\in\mathbb{C}[X_1,\dots,X_{r-1}]\) if \(r\geq j,\) \(A_{k}\in\mathbb{C}[X_1,\dots,X_{\sigma_i-1}]\) and \(c_{k,r}\in\mathbb{C}\).
	\medskip
	
	\noindent \textbf{Step 2}: We consider the polynomials \(U^{i}_{j}\) as polynomials in the variable \(T\) and we observe that \(A_k=\operatorname{lc}_T(U^{k}_{\sigma_k})\) the leading coefficient as polynomial in \(T\). Set \(\omega_k:=U^{k}_{\sigma_k}\) for \(k=0,\dots,s.\) By construction, we have that \(\ord_T(\omega_k)=g_k\) if  \(A_k\neq0\) for \(k=0,\dots,s\) and for all \(j\in(\sigma_{k-1},\sigma_k)\) we have \(\operatorname{lc}_T(U^{k}_{j})=0.\) We want to show that the system of polynomial equations defined by 
	\begin{equation}\label{eqn:sistemaec}
		\left\lbrace \begin{array}{ll}
			A_k\neq 0,&\text{for all \(k\)}\\
			\\
			\operatorname{lc}_T(U^{k}_{j})=0,&\text{for all \(k\) and all}\; j\in(\sigma_{k-1},\sigma_k)\\
			\\
			\operatorname{lc}_T(U^{s+1}_{\sigma_{s+1}})=0.
			
		\end{array} \right. 
	\end{equation}
	has a non-trivial compatible solution \((a_1,\dots,a_b)\in\mathbb{C}^b.\)
	\medskip
	
	\noindent \textbf{Step 3}: Finally, let us take a solution \((a_1,\dots,a_b)\in\mathbb{C}^b\) of the system (\ref{eqn:sistemaec}). We can consider the ring morphism defined by 
	\[
	\begin{array}{ccc}
		\mathbb{C}[X_1,\dots,X_b,T]&\xrightarrow{ev}& \overline{R}=\mathbb{C}\{t\}\\
		X_i&\mapsto& a_{i}\\
		T&\mapsto& t.
	\end{array}
	\]
	
	Therefore, we can define the germ $\xi$ of plane curve singularity given by the following Puiseux parameterization:
	\[
	\xi:\left\{\begin{array}{ll}
		x(t):=&t^\alpha\\
		y(t):= &ev(Y)=t^\beta+\displaystyle\sum_{i=1}^{b}a_it^{i+\beta}.
	\end{array}
	\right.
	\]
	
	If \(R\) stands for the local ring of the curve \(C\), then \(\Gamma=v(R)\). Moreover, it is easy to check that, by construction, the set \(\{ev(\omega_k)\}\) is a minimal set of generators of the \(R\)-module \(R+zR\). Therefore, we have \(v(R+zR)=L\), since \(\ord_T(\omega_k)=\ord_t(ev(\omega_k))\).
	
	\medskip
	
	We conclude proving Steps 1 and 2. 
	\medskip
	
	\noindent \textit{Proof of Step 1}: we apply induction. Define
	\begin{equation*}
		\begin{split}
			U^{0}_{\sigma_0}:=1\quad U^{1}_{\sigma_1}:=T^{g_1}(c_0+\sum_{i\geq 1} c_iX_iT^i)
		\end{split}
	\end{equation*}
	It is easily checked that \(U^{0}_{\sigma_0},U^{1}_{\sigma_1}\) are of the required form. 
	Let be \(\epsilon=e_1\alpha+e_2\beta \in \Gamma\), and write \(P(\epsilon):=T^{e_1\alpha}Y^{e_2}.\) Let us assume that for \(i<k<s\) there exists this family of polynomials. We are going to construct \(\{U^{k}_{j}\}\) for \(j\in I_k\); first, we define \(U^{k}_{\sigma_{k-1}}:=P(u_{k-1}-g_{k-1})U^{k-1}_{\sigma_{k-1}}\). Since by induction hypothesis \(U^{k-1}_{\sigma_{k-1}}\) is of the desired form, so is \(U^{k}_{\sigma_{k-1}}\). Now, for \(j\leq \sigma_{k}\) we define the corresponding \(U^{k}_{j}\) recursively:
	\begin{itemize}
		\item[$\diamond$] If \(h_k+j\notin E_{k-1}\) we put 
		$$
		U^{k}_{j+1}:=U^{k}_{j}-\operatorname{lc}_T(U^{k}_{j})T^{h_i+j}.
		$$
		\item[$\diamond$] If \(h_k+j\in E_{m}\cap(\mathbb{N}\setminus E_{m})\) for some \(m<k\) we set
		$$
		U^{k}_{j+1}:=\operatorname{lc}_T(U^{m}_{\sigma_{m}})U^{k}_{j}-\operatorname{lc}_T(U^{k}_{j})U^{m}_{\sigma_{m}}.
		$$
	\end{itemize}
	Finally, it is a straightforward computation to check that in both cases \(U^{k}_{j+1}\) has the desired form. 
	\medskip
	
	\noindent \textit{Proof of Step 2}: Observe that, by construction, 
	\[
	A_k=\prod_{j<k}A_jX_{\sigma_{k}}+V_{\sigma_k}\quad\text{with}\;V_{\sigma_k}\in\mathbb{C}[X_1,\dots,X_{\sigma_k-1}]
	\]
	so that the condition \(A_k\neq 0\) is equivalent to \(X_{\sigma_k}\neq\frac{V_{\sigma_k}}{\prod_{j<k}A_j}\). Also observe that for \(\ell\in(\sigma_{k-1},\sigma_k)\) by definition \[\operatorname{lc}_T(U^{k}_{\ell})=\prod_{j<k}c_{j,\ell}A_jX_{\ell}+V_{\ell}\quad\text{with}\;V_{\ell}\in\mathbb{C}[X_1,\dots,X_{\ell-1}]\]
	so that \(\operatorname{lc}_T(U^{k}_{\ell})=0\) is equivalent to \(X_{\ell}=\frac{V_{\ell}}{\prod_{j<k}c_{j,\ell}A_j}.\)
	
	Therefore, the system (\ref{eqn:sistemaec}) can be rewritten as
	\begin{equation}\label{eqn:sistemabien}
		\left\lbrace \begin{array}{ll}
			X_{\sigma_k}\neq\frac{V_{s_k}}{\prod_{j<k}A_j}&\text{for all \(k=2,\dots,s\)}\\
			\\
			X_{\ell}=\frac{V_{\ell}}{\prod_{j<k}c_{j,\ell}A_j}&\text{for all \(k=2,\dots,s+1\) and for all}\; \ell\in(\sigma_{k-1},\sigma_k)\\
			\\
			X_{\sigma_{s+1}}=\frac{V_{\sigma_{s+1}}}{\prod_{j<s+1}c_{j,\sigma_{s+1}}A_j}.
		\end{array} \right.
	\end{equation}

	Finally, observe that \((\sigma_{i-1},\sigma_i)\cap(\sigma_{j-1},\sigma_j)=\emptyset\) if \(i\neq j\). Since every isolated variable in the system is different, we can solve the system in the following recursive way. We start with \(A_0,A_1\in\mathbb{C}^{\ast}\), and for \(\ell\in(\sigma_1,\sigma_2)=(0,g_2-u_1)\) we have \(X_{\ell}=0\). Thus, \(X_{\sigma_2}\neq 0\). Let us denote \(a_{\sigma_2}=X_{\sigma_2}\in\mathbb{C}^{\ast}\). After that, since \(V_{\ell}\) depends on variables of lower index than \(\ell\) a recursive reasoning solves the system.
\end{proof}

The systems (\ref{eqn:sistemaec}) and (\ref{eqn:sistemabien}) have more equations than needed in order to obtain the required semimodule; the reason is that, if \(j\in(\sigma_{k-1},\sigma_k)\) is such that \(h_k+j \in E_{k-1}\), then \(U_{j}^k\) cannot be an element with \(v(U_j^k)=g_{k}\), since \(v(U_j^k)=h_k+j\in E_{k-1}\). Moreover, the last condition provides an element with \(\sigma_{s+1}+h_{s+1}=c_{s+1}+\alpha\beta\). If we apply Lemma \ref{cslemma}, it is trivial to see that \(\mathbb{N}+c_{s}+\alpha\beta\subset E_{s}\), since \(\mathbb{N}+c_{s}+\alpha\beta\subset\mathbb{N}+\overline{u}_{s}\); thus we can eliminate the last condition as well. As a result of that we may replace the system of equations given in the proof with the following:

\begin{equation}\label{eqn:estratification}
	\left\lbrace \begin{array}{ll}
		X_{\sigma_k}\neq\frac{V_{\sigma_k}}{\prod_{j<k}A_j}&\text{for all \(k=2,\dots,s\)}\\
		\\
		X_{\ell}=\frac{V_{\ell}}{\prod_{j<k}c_{j,\ell}A_j}&\text{for \(k=2,\dots,s+1\) and for all}\; \ell\in(\sigma_{k-1},\sigma_k)\\
		&\;\text{such that}\;\; h_k+\ell \in\mathbb{N}\setminus E_{k-1}
	\end{array} \right.
\end{equation}

\begin{rem}
	We have programmed with \textsf{Mathematica} software the computations of all possible increasing \(\Gamma\)--semimodules with \(g_1=\beta-\alpha\) and the corresponding system of Equations \ref{eqn:sistemabien}, see \cite[Appendix A]{almtesis}. 
\end{rem}
\begin{rem}\label{rem:deformation}
	A result of Zariski \cite[Chap. VI Prop. 2.1]{zariski} shows that any plane branch with one Puiseux pair is isomorphic to a deformation of the monomial curve given by \(t\mapsto (t^{\alpha},t^\beta)\). The system of equations (\ref{eqn:estratification}) together with the proof of Theorem \ref{thm:existence-curve} show that if \((a_1,\dots,a_b)\in\mathbb{C}\) is a solution of this system, then the increasing \(\Gamma\)-semimodule \(L\) associated to the equation system (\ref{eqn:estratification}) induces the deformation of the parameterization of the monomial curve defined by
	\[\left\{\begin{array}{cl}
		x(t):=&t^\alpha\\
		y(t):=&t^\beta+\displaystyle\sum_{i=1}^{b}a_it^{i+\beta}.
	\end{array}
	\right.
	\]
\end{rem}

\section{Revisiting K\"ahler differentials for one Puiseux pair}\label{sec: diff}

Recall that to any \(\mathbb{C}\)-algebra \(A\) we can associate the \(A\)-module \(\Omega_{A}:=I/I^2\), where \(I\) is the kernel of the diagonal surjection \(A\bigotimes_\mathbb{C} A\rightarrow A\); we call \(\Omega_{A}\) the module of K\"{a}hler differentials. The injection \(R\hookrightarrow \overline{R}\) induces a morphism of \(R\)-modules \(\varphi:\Omega_{R}\rightarrow\Omega_{\overline{R}}\). Now, the \(\overline{R}\)-module \(\Omega_{\overline{R}}\) is free of rank \(1\) and it is generated by \(dt:=t\otimes1-1\otimes t\). On the other hand, the \(R\)-module \(\Omega_{R}\) has two generators, namely
$$
dx=\alpha t^{\alpha-1} \ \ \   \mbox{  and    }  \ \ \  dy=\beta t^{\beta-1}+\sum_{j\geq \beta}ja_jt^{j-1}.
$$
In this way we can see that \(\varphi(\Omega_{R})\) is a free sub-\(\overline{R}\)-module of rank \(1\) generated by \(\alpha t^{\alpha-1}\). 
\medskip

We may define the set of values of the module of K\"{a}hler differentials as \(\Delta':=v(\varphi(\Omega_{R})).\) It is trivial to check that \(\Delta'\) is a non-normalized \(\Gamma\)--semimodule. Since \(\min\Delta'=\alpha-1\), its normalization is defined as 
\[
\Delta=(\Delta')^{\circ}=v\Big (R+R\frac{dy}{dx}\Big ).
\] 
From now on, we will deal with the normalized set of values \(\Delta.\) Let us recall a few notation from Delorme's paper \cite{delorme78}. Let us denote by \(g_0=0<g_1<\cdots< g_s\) the minimal set of generators of \(\Delta\) and \(g_{s+1}=\infty>g_s\). Also set
\begin{equation*}
	E_i:=\bigcup_{0\leq j\leq i}(\Gamma+g_j), \  E_s:=\Delta\ \mbox{and } \  u_i:=\min\{E_{i-1}\cap (\Gamma+g_i).\}
\end{equation*}
Delorme proves the following result

\begin{theorem}\label{thm:valuesemimod} 
Consider a plane branch with one characteristic exponent. Under the previous notation, set \(w_0=1,\) \(w_1=dy/dx\) and \(w_{s+1}=0\). Then, for all \hbox{\(1\leq i\leq s\)} there exists a relation 
	
	\[\displaystyle \omega_{i+1}=\sum_{0\leq j\leq i}F_{j,i}\omega_{j}\] where \(F_{j,i}\in R,\)  \hbox{\(u_i=v(F_{i,i})+g_i=\inf_j\{v(F_{j,i})+g_j\}\)} and \(v(\omega_{i+1})=g_{i+1}.\)
\end{theorem}
\begin{proof}
	See \cite[Lemma 12 (b)]{delorme78}.
\end{proof}

Here it is necessary to make two important remarks. The first one is the fact that in this case \(g_1=\beta-\alpha.\) The second one is that, as a consequence of Theorem \ref{thm:valuesemimod}, the value set of K\"{a}hler differentials has the property that, for all \(0\leq i\leq s\) we have the inequality \(g_{i+1}>u_i.\) Moreover, it is easily checked that this property is no longer true if we increase the number of Puiseux pairs of the irreducible plane curve singularity.
\medskip

The combination of Delorme's Theorem \ref{thm:valuesemimod} and Theorem \ref{thm:existence-curve} provides that all combinatorially possible value set of  K\"ahler differentials can be realized. This fact together with the constructive proof of Theorem \ref{thm:existence-curve} yield some advantages with respect to some results existing in the literature.
\medskip

In the sequel, we will call \emph{\((\beta-\alpha)\)--increasing \(\Gamma\)-semimodule} to an increasing \(\Gamma\)-semimodule with \(g_0=0,\,g_1=\beta-\alpha.\)
Observe that the \(\Gamma\)--semimodule of values of  K\"{a}hler differentials \(\Delta\) is a particular example of a \((\beta-\alpha)\)--increasing \(\Gamma\)-semimodule. Recall that the set \(\Delta\) is an analytic invariant of the curve $\xi$, as deduced from \cite[\S 4]{delorme78}.
\medskip

Moreover, Hefez and Hernandes \cite[Theorem~2.1]{hefez2} prove the following general result:

\begin{theorem}\cite[Theorem~2.1]{hefez2}\label{thm:hefezhernandes}
	Let \(\Gamma=\langle \overline{\beta}_0,\dots,\overline{\beta}_g\rangle\) be the semigroup of values of a plane branch \(\xi\). Then, a Puiseux parameterization of \(\xi\) is analytically equivalent to either \((t^{\overline{\beta}_0},t^{\overline{\beta}_1})\) or  
	$$
	\Bigg(t^{\overline{\beta}_0},t^{\overline{\beta}_1}+t^\lambda+\sum_{\scriptstyle i>\lambda\atop\scriptstyle i\notin \Delta'-\overline{\beta}_0}a_i t^i\Bigg),
	$$
	where \(\lambda\) is its Zariski invariant and \(\Delta'\) is the non-normalized set of orders of differentials of the branch. Moreover, if \(\varphi\) and \(\varphi '\) are Puiseux parameterizations of the previous form with the same \(\Gamma\) and \(\Delta'\) then they are analytically equivalent if and only if there is \(r\in\mathbb{C}^{\ast}\) such that \(r^{\lambda-\overline{\beta}_1}=1\) and \(a_i=r^{i-\overline{\beta}_1}a'_i.\)
\end{theorem}

\begin{rem}
	Theorem \ref{thm:hefezhernandes} was already stated by Delorme in \cite[Proposition 6]{delorme78} under two restrictions, namely the conditions (CE) and (CU) of \cite[Proposition 6]{delorme78}) for the set \(\Delta'\). Therefore, Hefez and Hernandes' theorem \ref{thm:hefezhernandes} could be formulated in a rough sentences as: the restrictions of Delorme for the set \(\Delta'\) can be eliminated.
\end{rem}
From the computational point of view, Hefez and Hernandes \cite{hefez1} provided a method to compute all possible \(\Delta\) sets appearing in Theorem \ref{thm:hefezhernandes}. However, this method is essentially brute force in the computation of a standard basis with variable coefficients for the module of K\"ahler differentials. In the case of one Puiseux pair, our constructive Theorem \ref{thm:existence-curve} provides a method which actually computes all possible \(\Delta\) set with much less computational cost than the standard bases process implemented by Hefez and Hernandes. Moreover the following Corollary to our Theorem \ref{thm:existence-curve} shows that in fact all combinatorial possibilities for the attainable value sets of K\"ahler differentials are realizable.

\begin{corollary}\label{cor:kahlerall}
	Let \(\Gamma=\langle \alpha,\beta\rangle\) be a numerical semigroup. Then,
	\(L\) is an increasing \(\Gamma\)--semimodule with first non-zero minimal generator equal to \(\beta-\alpha\) if and only if \(L\) can be realized as the set of values of K\"{a}hler differentials of some irreducible plane curve singularity with value semigroup \(\Gamma.\)
\end{corollary}
\begin{proof}
	This is a straightforward consequence of Theorem \ref{thm:valuesemimod} and Theorem \ref{thm:existence-curve}. 
\end{proof}

\begin{ex}

	Let us consider the numerical semigroup \(\Gamma=\langle 4,9\rangle.\) Following the tree structure explained in Section \ref{subsec:constmaltes}, all possible \(\beta-\alpha\)--increasing \(\Gamma\)--semimodules have the following tree structure:
	
	\begin{figure}[H]
		\begin{center}
			\begin{tikzpicture}[line cap=round,line join=round,>=triangle 45,x=1.0cm,y=1.0cm, scale=1.5]
				\clip(-3.,0.) rectangle (7.,5.);
				\draw [line width=1.pt] (2.5,3.5)-- (1.,2.5);
				\draw [line width=1.pt] (2.5,3.5)-- (2.,2.5);
				\draw [line width=1.pt] (2.5,3.5)-- (3.,2.5);
				\draw [line width=1.pt] (2.5,3.5)-- (4.,2.5);
				\draw [line width=1.pt] (3,2.5)-- (3.,1.5);
			 width=1.pt] (2.,2.5)-- (2.,2.);

				\begin{scriptsize}
					\draw [fill=black] (2.5,3.5) circle (2.pt);
					\draw [fill=black] (4.,2.5) circle (1.5pt);
					\draw [fill=black] (3,2.5) circle (1.5pt);
					\draw [fill=black] (2.,2.5) circle (1.5pt);
					\draw [fill=black] (1,2.5) circle (1.5pt);

					\draw [fill=black] (3.,1.5) circle (1.5pt);
				
					
					\node [below right][black] at (5.,4.3) {\tiny Added generators};
					\node [below right][black] at (5.25,3.5) {$\leftarrow g_1=5$};
					\node [below right][black] at (5.25,2.5) {$\leftarrow g_2$};
					\node [below right][black] at (5.25,1.5) {$\leftarrow g_3$};

			\node [below right][black] at (2.35,4.) {$5$};
		\node [below right][black] at (0.8,2.35) {$10$};
			\node [below right][black] at (1.8,2.35) {$11$};
				\node [below right][black] at (2.65,2.35) {$15$};
					\node [below right][black] at (3.8,2.35) {$19$};

					\node [below right][black] at (-2.1,4.3) {$\Delta$};
					\node [below right][black] at (-2.4,3.5) {$[0,5]$};
					\node [below right][black] at (-2.4,2.5) {$[0,5,g_2]$};
					\node [below right][black] at (-2.4,1.5) {$[0,5,g_2,g_3]$};

				\end{scriptsize}
			\end{tikzpicture}
			\vspace{-0.3cm}
			\caption{Tree of increasing \(\Gamma\)-semimodules with first non zero generator \(g=5\).}
		\end{center}
	\end{figure}
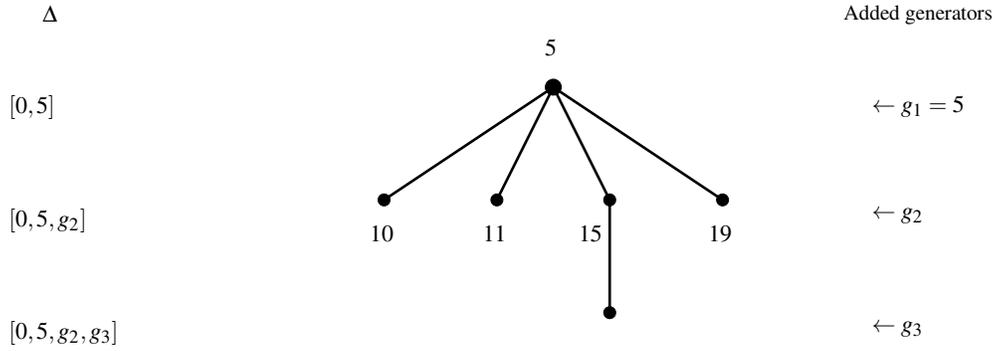
	
	According with Zariski \cite{zariski}, we can always consider a short Puiseux series of the form:
	\[y(t)=t^9+c_{10}t^{10}+c_{11}t^{11}+c_{15}t^{15}+c_{19}t^{19}.\]
	
	Our program \cite[Appendix A]{almtesis} computes all the associated conditions according to the proof of Theorem \ref{thm:existence-curve}. Those conditions read as follows:
	\begin{enumerate}
		\item [$\bullet$] Conditions for the K\"{a}hler semimodule \(\{0,5\}:\) \[c_{10}=c_{11}=c_{15}=c_{19}=0.\]
		
			\item [$\bullet$] Conditions for the K\"{a}hler semimodule \(\{0,5,10\}:\) \[c_{10}\neq 0\quad \text{and}\quad c_{11}=(19/18)c_{10}^2.\]
			
				\item [$\bullet$] Conditions for the K\"{a}hler semimodule \(\{0,5,11\}:\) \[c_{10}=0\quad \text{and}\quad c_{11}\neq 0.\]

\item [$\bullet$] Conditions for the K\"{a}hler semimodule \(\{0,5,15\}:\) \[c_{10}=c_{11}=0\quad\text{and}\quad c_{15}\neq 0.\]
					
\item [$\bullet$] Conditions for the K\"{a}hler semimodule \(\{0,5,19\}:\) \[c_{10}=c_{11}=c_{15}=0\quad \text{and}\quad c_{19}\neq 0.\]
						
		\item [$\bullet$] Conditions for the K\"{a}hler semimodule \(\{0,5,10,15\}:\) \[c_{10}\neq 0\quad \text{and}\quad c_{11}\neq (19/18)c_{10}^2.\]
		 Observe that this correspond with the generic component of the moduli space.
	\end{enumerate}
	
\end{ex}

\end{document}